\documentclass{amsart}

\usepackage[utf8x]{inputenc}
\usepackage[T1]{fontenc}

\usepackage{amsmath,amsfonts,amssymb,amsthm}

\usepackage{hyperref}

\usepackage{booktabs}

\usepackage{float}

\usepackage[shortlabels]{enumitem}

\usepackage[ruled,vlined]{algorithm2e}
    \DontPrintSemicolon
    \SetAlgoSkip{bigskip}

\theoremstyle{plain}   \newtheorem{theorem}{Theorem}[section]
\theoremstyle{plain}   \newtheorem{proposition}[theorem]{Proposition}
\theoremstyle{plain}   \newtheorem{corollary}[theorem]{Corollary}
\theoremstyle{plain}   \newtheorem{lemma}[theorem]{Lemma}
\theoremstyle{definition}   
\theoremstyle{definition}   \newtheorem{example}[theorem]{Example}

\newcommand{\nset}{\mathbb{N}}
\newcommand{\aset}{\mathcal{A}}
\newcommand{\anset}{\mathcal{A}_n}
\newcommand{\anmset}{\mathcal{A}_{n-1}}

\newcommand*{\psymb}{\mathrm{P}}

\newcommand*{\plac}{{\mathsf{plac}}}

\newcommand*{\hypo}{{\mathsf{hypo}}}
\newcommand*{\hypon}{{\mathsf{hypo}_n}}
\newcommand*{\hyponm}{{\mathsf{hypo}_{n-1}}}
\newcommand*{\hyponp}{{\mathsf{hypo}_{n+1}}}

\newcommand*{\sylv}{{\mathsf{sylv}}}

\usepackage{tikz}
\usetikzlibrary{arrows.meta,calc,decorations.pathmorphing,decorations.text,matrix,positioning}
\usetikzlibrary{shapes.geometric,shapes.misc,shapes.symbols}

\usetikzlibrary{matrix}

\tikzset{
  pretableaumatrix/.style={
    ampersand replacement=\&,
    matrix of math nodes,
    outer sep=1mm,
    inner sep=0mm,
    anchor=center,
    row sep={between borders,-\pgflinewidth},
    column sep={between borders,-\pgflinewidth},
    dottedentry/.style={densely dotted},
    spaceentry/.style={draw=none,execute at begin node=\null},
  },
  pretableaunode/.style={
    font=\small,
    draw=gray,
    sharp corners,
    rectangle,
    anchor=base,
    text height=3.75mm,
    text depth=1.25mm,
    minimum height=5mm,
    minimum width=5mm,
    inner sep=0mm,
    outer sep=0mm,
  },
  tableaumatrix/.style={
    pretableaumatrix,
    every node/.append style={
      pretableaunode,
    },
  },
  medtableaumatrix/.style={
    pretableaumatrix,
    every node/.append style={
      pretableaunode,
      font=\footnotesize,
      text height=2.75mm,
      text depth=.75mm,
      minimum height=3.5mm,
      minimum width=3.5mm
    },
  },
  smalltableaumatrix/.style={
    pretableaumatrix,
    every node/.append style={
      pretableaunode,
      font=\scriptsize,
      text height=1.85mm,
      text depth=.15mm,
      minimum height=2.5mm,
      minimum width=2.5mm,
    },
  },
  tinytableaumatrix/.style={
    pretableaumatrix,
    every node/.append style={
      pretableaunode,
      font=\tiny,
      text height=1.25mm,
      text depth=.15mm,
      minimum height=1.75mm,
      minimum width=1.75mm
    },
  },
  tableau/.style={
    baseline=-1.25mm,
    every matrix/.style={tableaumatrix},
  },
  medtableau/.style={
    baseline=-1.25mm,
    every matrix/.style={medtableaumatrix},
  },
  smalltableau/.style={
    baseline=-1.25mm,
    every matrix/.style={smalltableaumatrix},
  },
  preshapetableaumatrix/.style={
    pretableaumatrix,
    execute at end cell={\strut},
    every node/.append style={
      draw=black,
      anchor=base,
      inner sep=0mm,
      outer sep=0mm,
    },
    shadedentry/.style={fill=gray},
    darkshadedentry/.style={fill=darkgray},
  },
  medshapetableaumatrix/.style={
    preshapetableaumatrix,
    every node/.append style={
      text height=2.75mm,
      text depth=.75mm,
      minimum height=3.5mm,
      minimum width=3.5mm
    },
  },
  shapetableaumatrix/.style={
    ampersand replacement=\&,
    matrix of math nodes,
    outer sep=0mm,
    inner sep=0mm,
    anchor=base,
    row sep={between borders,-\pgflinewidth},
    column sep={between borders,-\pgflinewidth},
    execute at begin cell={\strut},
    every node/.append style={draw,anchor=base,text height=1mm,text depth=.5mm,minimum size=1.5mm,inner sep=0mm,outer sep=0mm},
  },
  shapetableau/.style={
    every matrix/.style={shapetableaumatrix},
  },
  topalign/.style={
    every matrix/.append style={name=maintableau,anchor=maintableau-1-1.base},
    baseline,
  },
}

\newcommand*\tableau[2][]{\tikz[tableau,#1]\matrix{#2};}

\usetikzlibrary{shapes.geometric} 
\usetikzlibrary{shapes.misc}  

\tikzset{
  bst/.style={
    standard/.style={
      font=\small,
      draw=gray,
      rounded rectangle,
      minimum width=4.5mm,
      minimum height=4.5mm,
      inner xsep=0mm,
      inner ysep=1mm,
      outer sep=0mm,
      line width=.5pt,
    },
    empty/.style={
      minimum width=3mm,
      minimum height=3mm,
    },
    triangle/.style={
      isosceles triangle,
      isosceles triangle apex angle=60,
      shape border rotate=90,
      rounded corners=2mm,
      minimum width=8mm,
      inner xsep=0mm,
      inner ysep=.5mm
    },
    blank/.style={
      draw=none,
    },
    nodecount/.style={
      blank,
      font=\scriptsize,
    },
    every node/.style={standard},
    every child/.style={draw=black,line width=.6pt},
    level distance=10mm,
    level 1/.style={sibling distance=60mm},
    level 2/.style={sibling distance=30mm},
    level 3/.style={sibling distance=15mm},
  },
  medbst/.style={
    bst,
    level distance=10mm,
    level 1/.style={sibling distance=15mm},
    level 2/.style={sibling distance=15mm},
    level 3/.style={sibling distance=15mm},
  },
  smallbst/.style={
    bst,
    level distance=8mm,
    level 1/.style={sibling distance=10mm},
    level 2/.style={sibling distance=10mm},
    level 3/.style={sibling distance=10mm},
  },
  tinybst/.style={
    bst,
    level distance=5mm,
    level 1/.style={sibling distance=8mm},
    level 2/.style={sibling distance=8mm},
    level 3/.style={sibling distance=8mm},
    every node/.append style={
      font=\footnotesize,
    },
    triangle/.append style={
      rounded corners=1mm,
      minimum width=7mm,
      inner xsep=-.5mm,
    },
  },
  microbst/.style={
    bst,
    standard/.append style={
      font=\scriptsize,
      minimum width=3mm,
      minimum height=3mm,
      inner ysep=.25mm,
    },
    level distance=3mm,
    level 1/.style={sibling distance=6mm},
    level 2/.style={sibling distance=6mm},
    level 3/.style={sibling distance=6mm},
  },
  nanobst/.style={
    bst,
    standard/.append style={
      font=\tiny,
      minimum width=2mm,
      minimum height=2mm,
      inner ysep=.25mm,
    },
    level distance=2mm,
    level 1/.style={sibling distance=4mm},
    level 2/.style={sibling distance=4mm},
    level 3/.style={sibling distance=4mm},
  },
}
\begin{document}

    \title{Identities and bases in the hypoplactic monoid}

    \author{Alan J. Cain}
    \address{
        Centro de Matem\'{a}tica e Aplica\c{c}\~{o}es\\
        Faculdade de Ci\^{e}ncias e Tecnologia\\
        Universidade Nova de Lisboa\\
        2829--516 Caparica\\
        Portugal
    }
    \email{a.cain@fct.unl.pt}

    \author{Ant\'{o}nio Malheiro}
    \address{
        Departamento de Matem\'{a}tica \& Centro de Matem\'{a}tica e Aplica\c{c}\~{o}es\\
        Faculdade de Ci\^{e}ncias e Tecnologia\\
        Universidade Nova de Lisboa\\
        2829--516 Caparica\\
        Portugal
    }
    \email{ajm@fct.unl.pt}

    \author{Duarte Ribeiro}
    \address{
        Departamento de Matem\'{a}tica \& Centro de Matem\'{a}tica e Aplica\c{c}\~{o}es\\
        Faculdade de Ci\^{e}ncias e Tecnologia\\
        Universidade Nova de Lisboa\\
        2829--516 Caparica\\
        Portugal
    }
    \email{dc.ribeiro@campus.fct.unl.pt}

    \thanks{
        This work is funded by National Funds through the FCT -- Funda\c{c}\~{a}o para a Ci\^{e}ncia e a Tecnologia, I.P., under the scope of the project {\scshape UIDB}/00297/2020 (Center for Mathematics and Applications) and the project {\scshape PTDC}/{\scshape MAT-PUR}/31174/2017.\\
        The third author is funded by National Funds through the FCT -- Funda\c{c}\~{a}o para a Ci\^{e}ncia e a Tecnologia, I.P., under the scope of the studentship {\scshape SFRH}/{\scshape BD}/138949/2018.
    }
    
    
    
    \begin{abstract}
        This paper presents new results on the identities satisfied by the hypoplactic monoid. We show how to embed the hypoplactic monoid of any rank strictly greater than $2$ (including infinite rank) into a direct product of copies of the hypoplactic monoid of rank $2$. This confirms that all hypoplactic monoids of rank greater than or equal to $2$ satisfy exactly the same identities. We then give a complete characterization of those identities, and prove that the variety generated by the hypoplactic monoid has finite axiomatic rank, by giving a finite basis for it. 
    \end{abstract}
    
	\maketitle

	
	\section{Introduction}
    \label{section:introduction}
	
	
    A (non-trivial) identity is a formal equality $u \approx v$, where $u$ and $v$ are words over some alphabet of variables, which is not of the form $u \approx u$. If a monoid is known to satisfy an identity, an important question is whether the set of identities it satisfies is finitely based, that is, if all these identities are consequences of those in some finite subset (see \cite{sapir_combinatorial,volkov_finitebasis}).
    
    The plactic monoid $\plac$, also known as the monoid of Young tableaux, is an important algebraic structure, first studied by Schensted \cite{Schensted1961} and Knuth \cite{knuth1970}, and later studied in depth by Lascoux and Schützenberger \cite{LS1978}. It is connected to many different areas of Mathematics, such as algebraic combinatorics, symmetric functions \cite{macdonald_symmetric}, crystal bases \cite{bump_crystalbases} and representation theory \cite{fulton_young}. In particular, the question of identities in the plactic monoid has received a lot of attention recently \cite{kubat_identities,izhakian_tropical}. 
    
    Finitely-generated polynomial-growth groups are virtually nilpotent and so satisfy identities \cite{gromov_growth}. In \cite{shneerson_identities}, it was possible to construct examples of finitely-generated polynomial-growth semigroups that do not satisfy identities, but it would be interesting to have a natural class of monoids with this property. One of the initial motivations for the study of the plactic monoids of finite rank, in this context, was to obtain a more natural example, although it is now known that they all satisfy identities. For example, the plactic monoid of rank $2$ satisfies Adjan's identity $xyyxxyxyyx \approx xyyxyxxyyx$, the shortest non-trivial identity satisfied by the bicyclic monoid \cite{adjan_identity}. The plactic monoid of rank $3$ satisfies the identity $uvvuvu = uvuvvu$, where $u(x,y)$ and $v(x,y)$ are respectively the left and right side of Adjan’s identity \cite{kubat_identities}. However, it does not satisfy Adjan's identity itself. 
    
    On the other hand, the plactic monoid of rank $n$ does not satisfy any non-trivial identity of length less than or equal to $n$ \cite{ckkmo_placticidentity}. Thus, there is no ``global'' identity satisfied by the plactic monoid of every rank and, consequently, the infinite-rank plactic monoid. Recently, a tropical representation for the plactic monoid of rank $n$ has been constructed \cite{johnson_kambites_preprint}, which implies that every plactic monoid of finite rank satisfies a non-trivial identity.
    
    In terms of identities, there is some relationship between plactic monoids, tropical algebra, and the bicyclic monoid: by \cite{daviaud_identities}, the monoid of $2 \times 2$ upper-triangular tropical matrices (see, for example, \cite{maclagan_tropical}), the bicyclic monoid, and the plactic monoid of rank $2$ satisfy precisely the same identities. The bicyclic monoid is not finitely based \cite{shneerson_axiomaticrank}, so none of these monoids are. This can be shown by using polytopes to analyze identities holding in the bicyclic monoid \cite{pastijn_polyhedral}.

    Monoids related to the plactic monoid, such as the hypoplactic monoid $\hypo$ (the monoid of quasi-ribbon tableaux), sylvester monoid $\sylv$ (the monoid of binary search trees \cite{hivert_sylvester}), Baxter monoid (pairs of twin binary search trees, connected with Baxter permutations \cite{giraudo_baxter2}) satisfy identities, and the shortest identities have been characterized \cite{cm_identities}. These identities are independent of rank (except for rank $1$, since the monoids of rank $1$ are all isomorphic to the free monogenic monoid and thus commutative). Also related to the plactic monoid, by its growth type \cite{duchamp_placticgrowth}, the Chinese monoid \cite{chinese_monoid} embeds into a direct product of copies of the bicyclic monoid \cite{jaszunska_chinese}. Furthermore, the Chinese and plactic monoids of rank $2$ coincide, hence, they satisfy the same identities.
    
    This paper focuses on the hypoplactic monoid $\hypo$, first studied in depth by Novelli \cite{novelli_hypoplactic}, which stands as the analogue of the plactic monoid in the theory of non-commutative symmetric functions and quasi-symmetric functions (see \cite{Krob1997,Krob1999}). It is also known as the monoid of quasi-ribbon tableaux, combinatorial objects which index quasi-ribbon functions and that are computed from words over the free monoid, using the Krob-Thibon insertion algorithm. The hypoplactic monoid arises by factoring the free monoid by the hypoplactic congruence $\equiv_{\hypo}$, which can be defined in several ways, one of which is $u \equiv_{\hypo} v$ if and only if $u$ and $v$ yield the same quasi-ribbon tableaux. Similarly to the plactic monoid and the crystal graph, in the sense of Kashiwara, the hypoplactic monoid has been shown to accept a quasi-crystal structure, which interacts neatly with the combinatorics of the quasi-ribbon tableaux and the hypoplactic version of the Robinson-Schensted-Knuth correspondence \cite{cm_hypo_crystal}.
    
    The main goal of this paper is to present a systematic study of the identities satisfied by the hypoplactic monoid. The paper is structured as follows: Section~\ref{section:preliminaries} provides the necessary background on universal algebra and the hypoplactic monoid. In Section~\ref{section:embeddings}, we first show that the hypoplactic monoid of rank $n$, denoted by $\hypon$, with $n \geq 2$, does not embed into the hypoplactic monoid of rank $2$. However, we show how to embed it into a direct product of copies of $\hypo_2$. We also show how to embed the infinite rank hypoplactic monoid $\hypo$ into a direct product of infinite copies of $\hypo_2$. Thus, we prove that all these monoids generate the same variety, denoted by $\mathbf{V}_{\hypo}$, and satisfy the same identities. We deduce the basis rank of $\mathbf{V}_{\hypo}$. Section~\ref{section:identities_and_bases} gives a complete characterization of the identities satisfied by $\hypo$, along with some immediate consequences. Furthermore, it also presents a finite basis for $\mathbf{V}_{\hypo}$, the variety generated by $\hypo$, thus showing it has finite axiomatic rank. 


	\section{Preliminaries and notation}
	\label{section:preliminaries}

	
	This section gives the necessary background on universal algebra (see \cite{bs_universal_algebra,mal2012algebraic,mckenzie2018algebras,Bergman_universal_algebra}), in the context of monoids, followed by the definition and essential facts about the hypoplactic monoid.
	
	For the necessary background on semigroups and monoids, see \cite{howie1995fundamentals}; for presentations, see \cite{higgins1992techniques}; for a general background on the plactic monoid, see \cite[Chapter~5]{lothaire_2002}. 
	

    \subsection{Varieties, identities and bases}
    \label{subsection:varieties_identities_and_bases}
        
    Let $X$ be a non-empty set, referred to as an \textit{alphabet}. The free semigroup over the alphabet $X$, denoted by $X^+$, is the set of all words over $X$, under the operation of word concatenation. If we include the empty word, denoted by $\varepsilon$, we obtain the free monoid over the alphabet $X$, denoted by $X^*$.
    
    For any word $u \in X^*$, the \textit{length} of $u$ is denoted by $|u|$, and for any $x \in X$, the number of occurrences of $x$ in $u$ is denoted by $|u|_x$. Suppose $u = u_1 \cdots u_k$, where $u_i \in X$. For any $1 \leq i \leq j \leq k$, the word $u_{i} \cdots u_{j}$ is a \textit{factor} of $u$. For any $i_1, \dots, i_m \in \{1, \dots, k\}$ such that $i_1 < \cdots < i_m$, the word $u_{i_1} \cdots u_{i_m}$ is a \textit{subsequence} of $u$.
    
    The \textit{content} of $u$, denoted by $c(u)$, is the infinite tuple of non-negative integers, indexed by $X$, whose $x$-th element is $|u|_x$. The \textit{support} of $u$, denoted by $supp(u)$, is the subset of $X$ such that $x \in supp(u) \iff \left|u\right|_x \geq 1$. Notice that if two words share the same content, then they also share the same support.

    A (\textit{monoid}) \textit{identity}, over an alphabet of variables $X$, is a formal equality $u \approx v$, where $u$ and $v$ are words in the free monoid $X^*$, and is non-trivial if $u \neq v$. We say a variable $x$ \textit{occurs} in $u \approx v$ if $x \in supp(u)$ or $x \in supp(v)$. The \textit{rank} of an identity is the number of distinct variables which occur in it. We say that an identity \textit{holds} in a monoid $M$ (or that $M$ satisfies the identity) if for every morphism $\psi: X^* \rightarrow M$ (also referred to as an \textit{evaluation}), the equality $\psi(u) = \psi(v)$ holds in $M$. In other words, $M$ satisfies the identity $u \approx v$ if equality in $M$ holds under every substitution of the variables of $u$ and $v$ by elements of $M$. 
    
    An identity $u \approx v$ is \textit{balanced} if $c(u)=c(v)$. Any identity satisfied by a monoid that contains a free monogenic submonoid, such as the aforementioned plactic and related monoids, must be balanced.
    
    Let $\Sigma$ be a nonempty set of identities, over an alphabet $X$. An identity $u \approx v$ is said to be a \textit{consequence} of $\Sigma$ if there exist $k \in \nset$, words $w_1, \dots, w_k \in X^*$, and substitutions $\sigma_1, \dots, \sigma_{k-1}$ of variables by elements of $X^*$ (that is, endomorphisms of $X^*$), such that $u = w_1$, $v = w_k$ and, for $1 \leq i < k$, 
    \begin{center}
    $w_i = r_i \sigma_i (p_i) s_i$, \\
    $w_{i+1} = r_i \sigma_i (q_i) s_i$,
    \end{center}
    for some $p_i, q_i, r_i, s_i \in X^*$ such that $p_i \approx q_i$ or $q_i \approx p_i$ is in $\Sigma$. 
    
    A class of monoids $\mathbf{K}$ is a \textit{monoid equational class} (or simply an \textit{equational class}) if there exists a set of identities $\Sigma$ such that $\mathbf{K}$ is the class of all monoids that satisfy all identities in $\Sigma$. Dually, a set of identities $\Sigma$ is an \textit{equational theory} if there exists a class of monoids $\mathbf{K}$ such that $\Sigma$ is the set of identities satisfied by all monoids in $\mathbf{K}$. When this holds, $\Sigma$ is called the \textit{equational theory of} $\mathbf{K}$. Equational theories are closed under taking consequences.
    
    On the other hand, a class of monoids $\mathbf{K}$ is a \textit{monoid variety} (or simply a \textit{variety}) if it is closed under the taking of homomorphic images, submonoids and direct products. For any class of monoids $\mathbf{K}$, we say that a variety is generated by $\mathbf{K}$ if it is the smallest variety containing $\mathbf{K}$, and denote it by $\mathbf{V} (\mathbf{K})$. Similarly, we say a variety $\mathbf{K}$ is generated by a monoid $M$ if it is the smallest variety containing $M$, and denote it by $\mathbf{V} (M)$. 
    
    The notions of equational class and variety coincide, as seen in the following theorem, known as Birkhoff's $HSP$ Theorem:
    \begin{theorem}[{\cite{birkhoff1935}}]
        $\mathbf{K}$ is an equational class if and only if $\mathbf{K}$ is a variety.
    \end{theorem}
    
    An important corollary of this theorem is that, for any monoid $M$, the identities satisfied by $M$ must also be satisfied by all other monoids in $\mathbf{V} (M)$.
    
    A set of identities $\mathcal{B}$ is an \textit{equational basis} (or simply \textit{basis}) of a variety $\mathbf{V}$ if the equational theory of $\mathbf{V}$ consists of all consequences of $\mathcal{B}$. A variety $\mathbf{V}$ is \textit{finitely based} if it admits a finite basis. The \textit{axiomatic rank} of $\mathbf{V}$ is the least natural number $r_a (\mathbf{V})$ such that $\mathbf{V}$ admits a basis $\mathcal{B}$, where the rank of each identity in $\mathcal{B}$ does not exceed $r_a (\mathbf{V})$. If no such natural number exists, we say that $\mathbf{V}$ has \textit{infinite axiomatic rank}. Notice that if $\mathbf{V}$ is finitely based, then it has finite axiomatic rank.
    
    A variety $\mathbf{V}$ is always generated by its $\mathbf{V}$-free monoid $\mathcal{F}_{\omega} (\mathbf{V})$ over an infinite alphabet. However, $\mathbf{V}$ may also be generated by a $\mathbf{V}$-free monoid $\mathcal{F}_{n} (\mathbf{V})$ over a finite alphabet with $n$ letters, for some $n \in \nset$. In such a case, the least natural number $r_b (\mathbf{V})$ such that $\mathbf{V}$ is generated by $\mathcal{F}_{r_b (\mathbf{V})} (\mathbf{V})$ is called the \textit{basis rank} of $\mathbf{V}$. If $\mathbf{V}$ is generated by a monoid with a finite number of generators, the minimal such number coincides with the basis rank of $\mathbf{V}$.
    
    A monoid $M$ is \textit{residually finite} (or \textit{finitely approximable}) if it is embeddable in a direct product of a family of finite monoids. If a monoid is both residually finite and finitely presented, then the word problem is solvable for such a monoid (see, for example, \cite{evans1978}).  
        
	
	\subsection{The hypoplactic monoid}
    \label{subsection:hypo}
    
    This subsection gives a brief overview of the hypoplactic monoid and its related combinatorial object and insertion algorithm, as well as results from \cite{cm_identities}. For more information, see \cite{novelli_hypoplactic} and \cite{cm_hypo_crystal}.
    
    Let $\aset = \{ 1 < 2 < 3 < \cdots \}$ denote the set of positive integers, viewed as an infinite ordered alphabet, and let $\anset = \{ 1 < \cdots < n \}$ denote the set of the first $n$ positive integers, viewed as a finite ordered alphabet.

    A \textit{quasi-ribbon tableau} is a (finite) grid of cells, aligned so that the leftmost cell in each row is below the rightmost cell of the previous row, filled with letters from $\aset$, such that the entries in each row are weakly increasing from left to right, and the entries in each column are strictly increasing from top to bottom. An example of a quasi-ribbon tableau is
    \begin{center}  
        $\tableau{
        1 \& 2 \& 4 \\
          \&   \& 5 \\
          \&   \& 6 \& 6 \\
          \&   \&   \& 7 \& 8 \\
        }$ .
    \end{center}    
    Observe that the same letter cannot appear in two different rows of a quasi-ribbon tableau.
    
    The following algorithm allows us to insert a letter from $\aset$ into an existing quasi-ribbon tableau, in order to obtain a new quasi-ribbon tableau:
        
    \begin{algorithm}[H]
        \KwIn{A quasi-ribbon tableau $T$ and a letter $a \in \aset$.}     
        \KwOut{A quasi-ribbon tableau $T \leftarrow a$.}
        \BlankLine
        \textbf{Method:}\\
            \Indp\lIf{there is no entry in $T$ that is less than or equal to $a$,}{output the tableau obtained by creating a new cell, labelled with $a$, and gluing $T$ by its top-leftmost entry to the bottom of this new cell;}
            \lElse{let $x$ be the right-most and bottom-most entry of $T$ that is less than or equal to $a$. Separate $T$ in two parts, such that one part is from the top left down to and including $x$. Put a new entry $a$ to the right of $x$ and glue the remaining part of $T$ (below and to the right of $x$) onto the bottom of the new entry $a$.}
        \Indm Output the resulting tableau.
        \caption{\textit{Krob--Thibon algorithm}.}
    \end{algorithm}
    
    Let $u \in \aset^*$. Using the insertion algorithm above, we can compute a unique quasi-ribbon tableau $\psymb_{\hypo} (u)$ from $u$: we start with the empty tableau and insert the letters of $u$, one-by-one from left-to-right -- see Example~\ref{example:insertion_alg}. 
    
    \begin{example}
    \label{example:insertion_alg}
        Computing $\psymb_{\hypo}(12654768)$:
        \begin{align*}  
            \xrightarrow{1}\quad
            \tableau{
                1 \\
            }
            \quad\xrightarrow{2}\quad 
            \tableau{
                1 \& 2 \\
            }
            \quad\xrightarrow{6}\quad 
            \tableau{
                1 \& 2 \& 6 \\
            }
            \quad\xrightarrow{5}\quad 
            \tableau{
                1 \& 2 \& 5 \\
                  \&   \& 6 \\
            }
            \quad\xrightarrow{4}\quad
            \tableau{
                1 \& 2 \& 4 \\
                \&   \& 5 \\
                \&   \& 6 \\
            }
            \\[10pt]
            \xrightarrow{7}\quad
            \tableau{
            1 \& 2 \& 4 \\
                \&   \& 5 \\
                \&   \& 6 \& 7 \\
            }
            \quad\xrightarrow{6}\quad
            \tableau{
                1 \& 2 \& 4 \\
                \&   \& 5 \\
                \&   \& 6 \& 6 \\
                \&   \&   \& 7 \\
            }
            \quad\xrightarrow{8}\quad
            \tableau{
                1 \& 2 \& 4 \\
                \&   \& 5 \\
                \&   \& 6 \& 6 \\
                \&   \&   \& 7 \& 8 \\
            }
        \end{align*}
    \end{example}

    We define the relation $\equiv_{\hypo}$ on $\aset^*$ as follows: For $u,v \in \aset^*$,
    \begin{center}
        $u \equiv_{\hypo} v \iff \psymb_{\hypo}(u) = \psymb_{\hypo}(v)$.
    \end{center}
    This relation is a congruence on $\aset^*$, called the \textit{hypoplactic congruence}. The factor monoid $\aset^*/{\equiv_{\hypo}}$ is the infinite-rank \textit{hypoplactic monoid}, denoted by $\hypo$. The congruence $\equiv_{\hypo}$ naturally restricts to a congruence on $\anset^*$, and the factor monoid $\anset^*/{\equiv_{\hypo}}$ is the \textit{hypoplactic monoid of rank} $n$, denoted by $\hypon$.
    
    It follows from the definition of $\equiv_{\hypo}$ that each element $[u]_{\hypo}$ of $\hypo$ can be identified with the combinatorial object $\psymb_{\hypo} (u)$. 
    
    Recall that the content of $u$ describes the number of occurrences of each letter of $\aset$ in $u$. It is immediate from the definition of the hypoplactic monoid that if $u \equiv_{\hypo} v$, then $c(u) = c(v)$. Thus, we can define the content of an element of $\hypo$ as the content of any word which represents it. Furthermore, since $c(u) = c(v)$ implies that $supp(u) = supp(v)$, we can also define the support of an element of $\hypo$ as the support of any word which represents it.
    
    Notice that $\hypo_n$ is a submonoid of $\hypo$, for each $n \in \nset$, and, for $n,m \in \nset$, if $n \leq m$, then $\hypo_n$ is a submonoid of $\hypo_m$.

    Let $u \in \anset^*$. Suppose $supp \left( u \right) = \{a_1 < \dots < a_k\}$, for some $k \in \nset$. We say $u$ has an $a_{i+1}$-$a_i$ \textbf{\textit{inversion}}, for $1 \leq i \leq k-1$, if it admits $a_{i+1} a_i$ as a subsequence. In other words, when reading $u$ from left-to-right, there is at least an occurrence of $a_{i+1}$ before the last occurrence of $a_i$. Notice that we only consider inversions of consecutive elements of the support of $w$.
    
    \begin{example}
        The word $31214$ has $3$-$2$ and $2$-$1$ inversions, but no $4$-$3$ inversion. On the other hand, $21341$ has a $2$-$1$ inversion, but no $4$-$3$ nor $3$-$2$ inversions.
    \end{example}
	
	The following characterization of the hypoplactic monoid is a consequence of \cite[Subsection~4.2]{novelli_hypoplactic}. It arises as a corollary of Theorem~$4.18$ and Note~$4.10$ of \cite{novelli_hypoplactic}:
	
    \begin{proposition}
    \label{proposition:hypo_inversion_char}
        For $u,v \in \anset^*$, we have that $u \equiv_{\hypo} v$ if and only if $u$ and $v$ share exactly the same content and inversions.
	\end{proposition}
	
	Attending to the previous result, we say that an element $[u]_{\hypo}$ of $\hypo$ has an $a_{i+1}$-$a_i$ inversion if the word $u$ itself, and hence any other word in $[u]_{\hypo}$, has an $a_{i+1}$-$a_i$ inversion. This characterization will be extensively used throughout the rest of this paper. 
	
    The hypoplactic monoid can also be defined by the presentation $\left\langle \aset \mid \mathcal{R}_{\hypo} \right\rangle$, where
    \begin{align*}
        \mathcal{R}_{\hypo} =& \left\{ (acb,cab): a \leq b < c \right\}\\
        & \cup \left\{ (bac,bca): a < b \leq c \right\}\\
        & \cup \left\{ (cadb,acbd): a \leq b < c \leq d \right\}\\
        & \cup \left\{ (bdac,dbca): a < b \leq c < d \right\}.
    \end{align*}
    
    The first two defining relations are known as the \textit{plactic relations}, while the remaining two are known as the \textit{hypoplactic relations}. A presentation for the hypoplactic monoid of rank $n$, for some $n \in \nset$, can be obtained by restricting generators and relations of the above presentation to generators in $\anset$. Notice that these relations are length-preserving. As such, it is easy to see that the hypoplactic monoids are residually finite. 

	
	\section{Embeddings}
	\label{section:embeddings}
	
	
	In this section, we prove that the hypoplactic monoids of rank greater than or equal to $2$ satisfy the same identities. We do this by constructing embeddings of hypoplactic monoids of any rank greater than $2$ into direct products of copies of the hypoplactic monoid of rank $2$, as it is not possible to directly embed one into another. Thus, they generate the same variety and, by Birkhoff's Theorem, satisfy exactly the same identities. We also show that the basis rank of the variety generated by $\hypo$ is $2$.
	
	\subsection{Non-existence of embedding into a hypoplactic monoid of lesser rank}
    \label{subsection:no_embedding_hypon_hyponm}
    
    It is not possible to directly embed a hypoplactic monoid of finite rank into a hypoplactic monoid of lesser rank:
	
	\begin{proposition}
    \label{prop:no_embedding_hypon_hypom}
		For all $n > m \geq 1$, there is no embedding of $\hypon$ into $\hypo_m$.
	\end{proposition}
	
	\begin{proof}
		First of all, notice that $\hypo_1$ is isomorphic to the free monogenic monoid and $\hypon$ is noncommutative, for any $n \geq 2$. Thus, there is no embedding of $\hypon$ into $\hypo_1$.
		
		On the other hand, if there exists an embedding of $\hypon$ into $\hypo_m$, for some $n > m \geq 2$, then, since $\hypo_m$ is a submonoid of $\hyponm$, there must also exist an embedding of $\hypon$ into $\hyponm$. As such, we just need to prove that this second embedding cannot exist.
		
		Suppose, in order to obtain a contradiction, that there exists $n \geq 3$ such that we have an embedding $\phi: \hypon \longmapsto \hyponm$. Without loss of generality, suppose $n$ is the smallest positive integer in such conditions.
		
		Observe that $supp \left( \phi \left( \left[1 \cdots (n-1) \right]_{\hypon} \right) \right) = \anmset$, that is, the image of the product of all generators of $\hypon$, except for $n$, has all the possible letters of $\anmset$. Indeed, if $supp \left( \phi \left( \left[1 \cdots (n-1) \right]_{\hypon} \right) \right) \subsetneqq \anmset$, we would be able to construct an embedding from the submonoid $\hyponm$ of $\hypon$, generated by all generators of $\hypon$ except for $n$, into a submonoid of $\hyponm$ isomorphic to $\hypo_{n-2}$. This contradicts the minimality of $n$.
				
		Hence, $\phi \left( \left[1 \cdots (n-1) \right]_{\hypon}^2 \right)$ has all the possible inversions of letters of $\anmset$. If we multiply this element by any other element of $\hyponm$, either on the left or the right, we obtain the same result, by Proposition~\ref{proposition:hypo_inversion_char}. Thus, we have that 
		\[
		\phi \left( [n]_{\hypon} \cdot \left[1 \cdots (n-1) \right]_{\hypon}^2 \right) = \phi \left( \left[1 \cdots (n-1) \right]_{\hypon}^2 \cdot [n]_{\hypon} \right).
		\]
		
		On the other hand, we have that
		\[
		[n]_{\hypon} \cdot \left[1 \cdots (n-1) \right]_{\hypon}^2 \neq \left[1 \cdots (n-1) \right]_{\hypon}^2 \cdot [n]_{\hypon}.
		\]
        
        This contradicts our hypothesis that $\phi$ is injective. Hence, for all $n \geq 2$, there is no embedding of $\hypon$ into $\hyponm$. As such, there is no embedding of $\hypon$ into $\hypo_m$, for $n > m \geq 2$.
	\end{proof}
	
	\begin{corollary}
    \label{corollary:no_embedding_hypo_hypo}
		There is no embedding of $\hypo$ into $\hypon$, for any $n \in \nset$. 
	\end{corollary}
	
	\begin{proof}
		If such an embedding existed, for some $n \in \nset$, then, by restricting the embedding to the first $n+1$ generators of $\hypo$, we would obtain an embedding of $\hyponp$ into $\hypon$, which contradicts the previous proposition.
	\end{proof}
	
	
	\subsection{Embedding into a direct product of copies of the hypoplactic monoid of rank 2}
    \label{subsection:embedding_hypon_direct_product_hypo2}
	
	Although it is not possible to directly embed a hypoplactic monoid of finite rank $n$ into the hypoplactic monoid of rank $2$, for $n > 2$, we now show how to embed $\hypon$ into a direct product of copies of $\hypo_2$. We also show how to embed $\hypo$ into a direct product of infinitely many copies of $\hypo_2$. For simplicity, we shall denote by $M^k$ the direct product of $k$ copies of a monoid $M$.
	
	For any $i,j \in \aset$, with $i < j$, define a map from $\aset$ to $\hypo_2$ in the following way: For any $a \in \aset$,
	\begin{align*} 
        a &\longmapsto \begin{cases}
            [1]_{\hypo_2} & \text{if } a = i;\\
            [2]_{\hypo_2} & \text{if } a = j;\\
            [21]_{\hypo_2} & \text{if } i < a < j;\\
            \left[\varepsilon\right]_{\hypo_2} & \text{otherwise};
        \end{cases}
	\end{align*}
	and extend it to a homomorphism $\varphi_{ij}: \aset^* \longrightarrow \hypo_2$, in the usual way.
	
	Note that $\varphi_{ij} (w)$ is the hypoplactic class of the word obtained from $w$ by replacing any occurrence of $i$ by $1$; any occurrence of $j$ by $2$; any occurrence of an $a$, with $i < a < j$, by $21$; and erasing any occurrence of any other element.
	
	\begin{lemma}
	\label{lemma:phi_factor_homomorphism}
		$\varphi_{ij}$ factors to give a homomorphism $\varphi_{ij} : \hypo \longrightarrow \hypo_2$.
	\end{lemma}
	
	\begin{proof}
        Since $\hypo$ is given by the presentation $\left\langle \aset \mid \mathcal{R}_{\hypo} \right\rangle$, we just need to verify that both sides of the plactic and hypoplactic relations have the same image under $\varphi_{ij}$. Let $a,b,c,d \in \aset$. Assume, without loss of generality, that $\varphi_{ij}$ does not map any letter to $\left[\varepsilon\right]_{\hypo_2}$.
        
        If $a \leq b < c \leq d$, then either $\varphi_{ij}$ maps at least one letter to $[21]_{\hypo_2}$, or $\varphi_{ij}$ maps $a$ and $b$ to $[1]_{\hypo_2}$, and $c$ and $d$ to $[2]_{\hypo_2}$. Thus,
        \[
            \varphi_{ij} (acb) = \varphi_{ij} (cab) \quad \text{and} \quad \varphi_{ij} (acbd) = \varphi_{ij} (cadb),
        \]
        since both classes in each side of the equalities have $2$-$1$ inversions.

        If $a < b \leq c$, then either $\varphi_{ij}$ maps at least one letter to $[21]_{\hypo_2}$, or $\varphi_{ij}$ maps $a$ to $[1]_{\hypo_2}$, and $b$ and $c$ to $[2]_{\hypo_2}$. Thus,
        \[
            \varphi_{ij} (bac) = \varphi_{ij} (bca).
        \]
        since both classes in each side of the equality have $2$-$1$ inversions.
	    
        If $a < b \leq c < d$, then $\varphi_{ij}$ maps $b$ and $c$ to $[21]_{\hypo_2}$. Thus, by the same reasoning,
        \[
            \varphi_{ij} (bdac) = \varphi_{ij} (dbca).
        \]
        
		Hence, $\mathcal{R}_{\hypo} \subseteq \ker \varphi_{ij}$.
	\end{proof}
	
	Let $w \in \anset^*$, for some $n \geq 3$. Suppose $supp \left( w \right) = \{a_1 < \dots < a_k\}$, for some $k \in \nset$. Observe that, ranging $1 \leq i < k$, we can get from the maps $\varphi_{a_i a_{i+1}}$ the number of occurrences of $a_i$ and $a_{i+1}$ in $w$. Furthermore, we can also check if $w$ has an $a_{i+1}$-$a_i$ inversion: Since no element $a \in \aset$ such that $a_i < a < a_{i+1}$ occurs in $w$, each occurrence of $1$ in (a word in) $\varphi_{a_i a_{i+1}}([w]_{\hypo})$ corresponds to an occurrence of $a_i$ in $w$ and, similarly, each occurrence of $2$ in $\varphi_{a_i a_{i+1}}([w]_{\hypo})$ corresponds to an occurrence of $a_{i+1}$ in $w$. Thus, $\varphi_{a_i a_{i+1}}([w]_{\hypo})$ has a $2$-$1$ inversion if and only if $w$ has an $a_{i+1}$-$a_i$ inversion. Hence, we get the following lemma:
	
	\begin{lemma}
    \label{lemma:varphi_equiv_hypon}
		Let $u,v \in \anset^*$. Then, $u \equiv_{\hypo} v$ if and only if $\varphi_{ij}([u]_{\hypo}) = \varphi_{ij}([v]_{\hypo})$, for all $1 \leq i < j \leq n$.
	\end{lemma}
	
	\begin{proof}
	    The direct implication is trivial, since $\varphi_{ij}$ is well-defined as a map, for all $1 \leq i < j \leq n$. The proof of the converse follows from the previous observations, as well as Proposition~\ref{proposition:hypo_inversion_char}.
	\end{proof}

	For each $n \in \nset$, with $n \geq 3$, let $I_n$ be the index set
	\[
	    \left\{ (i,j): 1 \leq i < j \leq n \right\},
	\]
	and let $I := \bigcup_{n \in \nset} I_n$. Now, consider the map
	\[
	\phi_n : \hypon \longrightarrow \prod\limits_{I_n} \hypo_2,
	\]
	whose $(i,j)$-th component is given by $\varphi_{ij}([w]_{\hypo})$, for $w \in \anset^*$ and $(i,j) \in I_n$.
		
	\begin{proposition}
    \label{prop:hypon_embed_hypo2}
		The map $\phi_n$ is an embedding.
	\end{proposition}
	
	\begin{proof}
		It is clear that $\phi_n$ is a homomorphism. It follows from the definition of $\phi_n$ and Lemma~\ref{lemma:varphi_equiv_hypon} that, for any $u,v \in \anset^*$, we have $u \equiv_{\hypon} v$ if and only if $\phi_n([u]_{\hypo}) = \phi_n([v]_{\hypo})$, hence $\phi_n$ is an embedding.
	\end{proof}
	
	Thus, for each $n \in \nset$, we can embed $\hypon$ into a direct product of $\binom{n}{2}$ copies of $\hypo_2$.
	
	Similarly, we can embed $\hypo$ into a direct product of infinitely many copies of $\hypo_2$. Consider the map
	\[
	\phi : \hypo \longrightarrow \prod\limits_{I} \hypo_2,
	\]
	whose $(i,j)$-th component is given by $\varphi_{ij}([w]_{\hypo})$, for $w \in \aset^*$ and $(i,j) \in I$.
	
	\begin{proposition}
    \label{prop:hypo_embed_hypo2}
		The map $\phi$ is an embedding.
	\end{proposition}
	
	\begin{proof}
		It is clear that $\phi$ is a homomorphism. Notice that, for any word $w \in \aset^*$, there must exist $n \in \nset$ such that $w \in \anset^*$. Furthermore, for $(i,j) \in I_n$, we have that the $(i,j)$-th component of $\phi([w]_{\hypo})$ is equal to the $(i,j)$-th component of $\phi_n([w]_{\hypo})$. Thus, for $u,v \in \aset^*$, if $\phi([u]_{\hypo}) = \phi([v]_{\hypo})$, then $\phi_n([u]_{\hypo}) = \phi_n([v]_{\hypo})$, for some $n \in \nset$ such that $u,v \in \anset^*$.
		
		It follows from Lemma~\ref{lemma:varphi_equiv_hypon} that, for any $u,v \in \aset^*$, we have $u \equiv_{\hypo} v$ if and only if $\phi([u]_{\hypo}) = \phi([v]_{\hypo})$, hence $\phi$ is an embedding.
	\end{proof}
		
	As such, all hypoplactic monoids of rank strictly greater than $2$ are in the variety generated by $\hypo_2$. Since $\hypo_2$ is a submonoid of $\hypo$ and $\hypon$, for any $n \geq 3$, they all generate the same variety, which we will denote by $\mathbf{V}_{\hypo}$. Thus, by Birkhoff's Theorem, we have the following result:
	
	\begin{theorem}
    \label{theorem:hypo-id-hypo_2}
		For any $n \geq 2$, $\hypo$ and $\hypon$ satisfy exactly the same identities.
	\end{theorem}
	
	Another consequence of $\mathbf{V}_{\hypo}$ being generated by $\hypo_2$ is the following:
	
	\begin{proposition}
	    The basis rank of $\mathbf{V}_{\hypo}$ is $2$.
	\end{proposition}
	
	\begin{proof}
	    Since $\mathbf{V}_{\hypo}$ is generated by $\hypo_2$, and $\hypo_2$ is defined by a presentation where the alphabet has two generators, then $r_b \left(\mathbf{V}_{\hypo}\right)$ is less than or equal to $2$.
	    
	    On the other hand, notice that any monoid generated by a single element is commutative. Since $\hypo$ is not commutative, $\mathbf{V}_{\hypo}$ cannot be generated by any monoid which is itself generated by a single element. As such, $r_b \left(\mathbf{V}_{\hypo}\right)$ is strictly greater than $1$.
	    
	    Hence, the basis rank of $\mathbf{V}_{\hypo}$ is $2$.
	\end{proof}

	
	\section{Identities and bases}
	\label{section:identities_and_bases}
	
	
	In this section, we obtain a complete characterization of the identities satisfied by $\hypo$, a finite basis for $\mathbf{V}_{\hypo}$ and also its axiomatic rank. 
	
	\subsection{Characterization of the identities satisfied by the hypoplactic monoid}
    \label{subsection:identities_hypo}
	
	Due to Theorem \ref{theorem:hypo-id-hypo_2}, the identities satisfied by $\hypo$ and those satisfied by $\hypo_2$ are exactly the same. We shall use $\hypo_2$ to obtain a characterization of the identities satisfied by $\hypo$.
		
	Observe that, for each element of $\hypo_2$, there is at most one other distinct element of $\hypo_2$ which has the same content as it. Indeed, for an element of $\hypo_2$, with fixed content and support $\{1,2\}$, either it has a $2$-$1$ inversion or not. 
	
	\begin{theorem}
    \label{theorem:hypo_identities}
		The identities $u \approx v$ satisfied by $\hypo$ are exactly the balanced identities such that, for any variables $x,y$ that occur in $u$ and $v$, $u$ admits $xy$ as a subsequence if and only if $v$ does too.
	\end{theorem}
	
	\begin{proof}		
		We first prove by contradiction that an identity satisfied by $\hypo_2$ must satisfy the previously mentioned conditions. Suppose $u \approx v$ is an identity satisfied by $\hypo_2$. Since $\hypo$ contains the free monogenic submonoid, we know that any identity satisfied by $\hypo$ must be a balanced identity. Thus, we assume $u \approx v$ is a balanced identity. 
		
		Suppose, in order to obtain a contradiction, that there exist variables $x,y$ occurring in $u$ and $v$, such that $u$ admits $xy$ as a subsequence, but $v$ does not. Observe that, since both $x$ and $y$ occur in $v$, then $v$ must admit $yx$ as a subsequence. 
		
		Then, taking the evaluation $\psi$ of $X$ in $\hypo_2$ such that $\psi (x) = [2]_{\hypo_2}$, $\psi (y) = [1]_{\hypo_2}$ and $\psi (z) = [\varepsilon]_{\hypo_2}$, for all other variables $z$, we have
		\[
		\psi (u) = [2^{\alpha} 1^{\beta}]_{\hypo_2} \text{ and } \psi (v) = [1^{\beta} 2^{\alpha}]_{\hypo_2}.
		\]
		
		Thus, by Proposition \ref{proposition:hypo_inversion_char}, we have that $\psi (u) \neq \psi (v)$, which contradicts our hypothesis that $u \approx v$ is an identity.
			
		We now prove by contradiction that an identity which satisfies the previously mentioned conditions must also be satisfied by $\hypo_2$. Suppose that $u \approx v$ is a balanced identity, such that, for all variables $x,y$ which occur in $u$ and $v$, $u$ admits $xy$ as a subsequence if and only if $v$ does too. Suppose, in order to obtain a contradiction, that there is some evaluation $\psi$ of $X$ in $\hypo_2$ such that $\psi(u) \neq \psi(v)$.
		
		Notice that, since $u \approx v$ is a balanced identity, then $\psi(u)$ and $\psi(v)$ have the same content. Since we assumed that $\psi(u) \neq \psi(v)$, then, by Proposition \ref{proposition:hypo_inversion_char}, we have that $supp \left( \psi(u) \right) = supp \left( \psi(v) \right) = \{ 1,2 \}$, and $\psi(u)$ has a $2$-$1$ inversion and $\psi(v)$ does not, or vice-versa.
		
		We assume, without loss of generality, that $\psi(u)$ has a $2$-$1$ inversion and $\psi(v)$ does not. Note that $\psi(v)$, as a word class, has only one word, of the form $1^{\alpha} 2^{\beta}$, for some $\alpha, \beta \geq 1$. Then, $v$ must be of the form $v = v_1 v_2$ or $v = v_1 z v_2$, with $z \in X$, where: for each variable $x$ which occurs in $v_1$, $\psi (x)$ has support $\{1\}$; for each variable $y$ which occurs in $v_2$, $\psi (y)$ has support $\{2\}$; and $\psi (z)$ has support $\{1,2\}$. 
		
		Notice that $z$ is a variable that occurs in neither $v_1$ nor $v_2$, and that no variable occurs simultaneously in $v_1$ and $v_2$. Also notice that, for any variables $x$, which occurs in $v_1$, and $y$, which occurs in $v_2$, $v$ admits $xz, xy$ and $zy$ as subsequences (if there is a variable $z$ in the previously mentioned conditions), but not $zx, yx$ nor $yz$. Thus, by the conditions imposed upon $u \approx v$, $u$ must be of the form $u = u_1 u_2$ or $u = u_1 z u_2$, where $c(u_1)=c(v_1)$ and $c(u_2)=c(v_2)$. Hence, $\psi (u_1) = \psi (v_1)$ and $\psi (u_2) = \psi (v_2)$, by the observations in the previous paragraph.
		
		Thus, we either have
		\[
		\psi(u) = \psi(u_1) \psi(u_2) = \psi(v_1) \psi(v_2) = \psi(v)
		\]
		or
		\[
		\psi(u) = \psi(u_1) \psi(z) \psi(u_2) = \psi(v_1) \psi(z) \psi(v_2) = \psi(v).
		\]
		
		Thus, we have reached a contradiction, hence, there is no evaluation $\psi$ of $X$ in $\hypo_2$ such that $\psi(u) \neq \psi(v)$. Thus, $u \approx v$ is an identity satisfied by $\hypo_2$.
	\end{proof}
	
	With this characterization, we recover as a corollary the following result:
	
	\begin{corollary}[{\cite[Proposition~12]{cm_identities}}]
	\label{corollary:cm_identities}
		The following non-trivial identities are satisfied by $\hypo$:
		\begin{center}
			$xyxy \approx xyyx \approx yxxy \approx yxyx$;\\
			$xxyx \approx xyxx$.
		\end{center}
		Furthermore, up to equivalence (that is, up to renaming variables or swapping both sides of the identities), these are the shortest non-trivial identities satisfied by $\hypo$.
	\end{corollary}
	
	We also easily obtain some examples of important non-trivial identities satisfied by $\hypo$:
	
	\begin{example}
	\label{example:hypo_basis_identities}
		The following non-trivial identities are satisfied by $\hypo$:
		\begin{align}
            xyzxty \approx yxzxty; \tag{L} \\
            xzxytx \approx xzyxtx; \tag{M} \\
            xzytxy \approx xzytyx. \tag{R}
        \end{align}
	\end{example}
	
    It is well known that the set of all balanced identities is the equational theory of the variety $\mathbf{Com}$ of all commutative monoids, which is generated by the free monogenic monoid. On the other hand, the set $J_2$ of all identities $u \approx v$ where $u$ and $v$ share exactly the same subsequences of length at most 2 is the equational theory of the pseudovariety $\mathbf{J}_{2}$, which, due to Eilenberg's correspondence (see \cite{Eilenberg_book76,pin1986varieties}), corresponds to the class of all piecewise testable languages of height $2$ (see \cite{simon_thesis}). This pseudovariety is generated by $\mathcal{C}_3$, the 5-element monoid of all order preserving and extensive transformations of the chain $1 < 2 < 3$ (see \cite{volkov_reflexive_relations}). Thus, the equational theory of $\mathbf{V} \left( \mathcal{C}_3 \right)$, the variety generated by $\mathcal{C}_3$, is $J_2$. It is easy to see that the equational theory of $\mathbf{V}_{\hypo}$ is the intersection of the set of all balanced identities and the set $J_2$. As such, we have the following corollary of Theorem~\ref{theorem:hypo_identities}, suggested by the anonymous referee:
    
	\begin{corollary}
	\label{corollary:varietal_join}
        $\mathbf{V}_{\hypo}$ is the varietal join $\mathbf{Com} \vee \mathbf{V} \left( \mathcal{C}_3 \right)$, and is generated by the free monogenic monoid and the monoid $\mathcal{C}_3$.
    \end{corollary}
	
	An alternative characterization of the identities satisfied by $\hypo$ is the following:
	
	\begin{corollary}
        The identities $u \approx v$ satisfied by $\hypo$ are balanced identities such that, for any variables $x,y$ that occur in $u$ and $v$, $u|_{x,y} \approx v|_{x,y}$ is an identity satisfied by $\hypo$, where $u|_{x,y}$ and $v|_{x,y}$ are obtained from $u$ and $v$, respectively, by eliminating every occurrence of a variable other than $x$ or $y$.
    \end{corollary}
    
    \begin{proof} 
        The proof of the implication is trivial, as it follows from the definition of an identity satisfied by a monoid.
        
        Suppose then that the identity $u \approx v$ is balanced and satisfies the following property: For any variables $x,y$ that occur in $u$ and $v$, $u|_{x,y} \approx v|_{x,y}$ is an identity satisfied by $\hypo$.
        
        By Theorem \ref{theorem:hypo_identities}, to prove that $u \approx v$ is satisfied by $\hypo$, we only need to show that $u$ admits $xy$ (or $yx$) as a subsequence if and only if $v$ does too, for any variables $x,y$ that occur in $u$ and $v$.
        
        Let $x,y$ be variables that occur in $u$ and $v$. Hence, by the hypothesis, we have that $u|_{x,y} \approx v|_{x,y}$ is an identity satisfied by $\hypo$. By Theorem \ref{theorem:hypo_identities}, we know that $u|_{x,y}$ admits $xy$ (or $yx$) as a subsequence if and only if $v|_{x,y}$ does too. 
        
        Notice that $u|_{x,y}$ is the unique subsequence of $u$ with the same number of occurrences of $x$ and $y$. Therefore, we can conclude that $u|_{x,y}$ admits $xy$ (or $yx$) as a subsequence if and only if $u$ does too. The same can be stated about $v|_{x,y}$ and $v$.
        
        Therefore, $u$ admits $xy$ (or $yx$) as a subsequence if and only if $v$ does too. Hence, we conclude that  $u \approx v$ is an identity satisfied by $\hypo$.
	\end{proof}	
	
	It is very easy to verify if a balanced identity, over a two-letter alphabet, is satisfied by $\hypo$, by the following complete characterization:
	
	\begin{corollary}
		The non-trivial identities, over the two-letter alphabet $\{x,y\}$, satisfied by $\hypo$ are balanced identities such that neither side of the identity is of the form $x^a y^b$ or $y^b x^a$, for some $a,b \in \nset$.
	\end{corollary}
	
	\begin{proof}
		Let us first suppose that $u \approx v$ be a non-trivial identity, over the two-letter alphabet $\{x,y\}$, satisfied by $\hypo$. By the previous theorem, we know that this identity must be balanced. Suppose, without loss of generality, that $u = x^a y^b$, for some $a,b \in \nset$. Then, since $u$ does not admit $yx$ as a subsequence, $v$ cannot as well. Since $u \approx v$ is a balanced identity, we conclude that $v = x^a y^b = u$, which contradicts our hypothesis that $u \approx v$ is non-trivial.
		
		Since this argument can be applied to all other possible cases, we conclude that neither side of the identity is of the form $x^a y^b$ or $y^b x^a$.
		
		Conversely, let $u \approx v$ be a non-trivial, balanced identity, over the two-letter alphabet $\{x,y\}$, such that neither side of the identity is of the form $x^a y^b$ or $y^b x^a$, for some $a,b \in \nset$. The, since $u \approx v$ is a non-trivial identity, both $x$ and $y$ must occur at least once in both $u$ and $v$. 
		
		Observe that the only words over $\{x,y\}$, where both $x$ and $y$ occur, that do not admit $yx$ as a subsequence, are words of the form $x^a y^b$. Similarly, the only words over $\{x,y\}$, where both $x$ and $y$ occur, that do not admit $xy$ as a subsequence, are words of the form $y^b x^a$. Thus, both $u$ and $v$ admit $xy$ and $yx$ as subsequences. Hence, by the previous Theorem, $u \approx v$ is satisfied by $\hypo$.
	\end{proof}
	
	The following corollary will be important in the next subsection:
	
	\begin{corollary}
    \label{corollary:hypo_shortest_identity_n_variables}
		The shortest non-trivial identity, with $n$ variables, satisfied by $\hypo$, is of length $n+2$.
	\end{corollary}
	
	\begin{proof}
		It is immediate, by the previous theorem, that for variables $x, a_1 \dots a_{n-1}$, 
		\[
		x a_1 \dots a_{n-1} x x \approx x x a_1 \dots a_{n-1} x
		\]
		is an identity satisfied by $\hypo$. 
		
		On the other hand, suppose that there exists a non-trivial identity $u \approx v$, with $n$ variables, satisfied by $\hypo$. Let $\mathcal{X}$ be the set of variables which occur in $u$ and $v$. By Theorem~\ref{theorem:hypo_identities}, if $u$ admits a subsequence $xy$, for variables $x,y \in X$, then $v$ must also admit such a subsequence. Thus, we easily conclude that the identity cannot be of length $n$, otherwise it would be trivial. As such, there must be some variable $z$ which occurs twice in $u$ and $v$.
		
		If $u \approx v$ is of length $n+1$, then $z$ occurs twice and all other variables $a_1, \dots, a_{n-1}$ occur only once, in $u$ and $v$. Assume, without loss of generality, that $u$ admits the subsequence $a_1 \cdots a_{n-1}$. Again, by Theorem~\ref{theorem:hypo_identities}, $v$ must also admit this subsequence.
		
		Let $i$ be the greatest index such that $u$ admits $a_i z$ as a subsequence, but not $z a_i$, and let $j$ be the smallest index such that $u$ admits $z a_j$ as a subsequence, but not $a_j z$. Since $z$ occurs exactly twice in $u$, we can conclude that 
		\[
		u = a_1 \cdots a_i \, z \, a_{i+1} \cdots a_{j-1} \, z \, a_{j} \cdots a_{n-1}.
		\]
		
		Once again, by Theorem~\ref{theorem:hypo_identities}, $v$ must also satisfy these conditions. Hence, we can conclude that $u=v$.
		
		This contradicts the hypothesis that $u \approx v$ is a non-trivial identity satisfied by $\hypo$. Hence, there is no non-trivial identity, with $n$ variables, satisfied by $\hypo$, of length $n+1$.
	\end{proof}
	
	
	\subsection{The axiomatic rank of the variety generated by the hypoplactic monoid}
    \label{subsection:basis_hypo}
	
	Now, we prove that not only $\mathbf{V}_{\hypo}$ has finite axiomatic rank, it is also finitely based. We give a basis for $\mathbf{V}_{\hypo}$ with three identities, all of them over a four-letter alphabet, each of length $6$. This basis is minimal, in the sense that no identity in this basis is a consequence of the others, and also that each identity is of minimal length, for identities satisfied by $\hypo$ over a four-letter alphabet. Furthermore, we also prove that there exists no basis for $\mathbf{V}_{\hypo}$ with only identities over an alphabet with at most three variables, thus showing that the axiomatic rank of $\mathbf{V}_{\hypo}$ is $4$.
	
	\begin{theorem}
		$\mathbf{V}_{\hypo}$ admits a finite basis $\mathcal{B}_{\hypo}$, with the following identities:
		\begin{align}
            xyzxty \approx yxzxty; \tag{L} \label{idL}\\
            xzxytx \approx xzyxtx; \tag{M} \label{idM}\\
            xzytxy \approx xzytyx. \tag{R} \label{idR}
        \end{align}
	\end{theorem}
	
	\begin{proof}
		Let $\mathcal{B}_{\hypo}$ be the set of identities which contains the previously mentioned identities. Notice that these identities are the ones given in Example \ref{example:hypo_basis_identities}.
		
		The following proof will be done by induction, in the following sense: First, we order identities by their length. Then, within a set of identities of the same length, we order them by the length of the common prefix of both sides of the identity. The induction will be first on the length of the identities, and then on the length of the suffix, that is, the length of the identity minus the length of the common prefix. 
		
		The base cases for the induction on the length of the identities are those given in Corollary \ref{corollary:cm_identities}, as they are, up to equivalence, the shortest non-trivial identities satisfied by $\hypo$. Notice that these identities are consequences of $\mathcal{B}_{\hypo}$: to show this, we just need to replace the variables $z$ and $t$ with the empty word, and, if necessary, rename the remaining variables.
		
		The base cases for the induction on the length of the suffix, for identities of length $n$ (with $n \geq 4$), are those identities of the form 
		\[
		w x y \approx w y x,
		\]
		where $w$ is a word of length $n-2$ and $x,y$ are variables. Observe that, since any identity $u \approx v$ satisfied by $\hypo$ is a balanced identity, there are no non-trivial identities, of length $n$, with a common prefix of length greater than $n-2$, satisfied by $\hypo$. Furthermore, since $wxy$ admits a subsequence $xy$ and $wyx$ admits a subsequence $yx$, then $x$ and $y$ must both occur in $w$. Thus, $w$ is of the form
		\[
		w_1 x w_2 y w_3 \quad \text{or} \quad w_1 y w_2 x w_3,
		\]
		for some words $w_1, w_2, w_3$. Therefore, by replacing $z$ with $w_2$, and $t$ by $w_3$, and, if necessary, renaming $x$ and $y$, we can immediately deduce this identity from the identity \eqref{idR} of $\mathcal{B}_{\hypo}$.
		
		The idea of the proof of the induction step is that, for any identity $u \approx v$, we can apply identities of $\mathcal{B}_{\hypo}$, finitely many times, to deduce a new identity $u^* \approx v$ from $u \approx v$, such that $u^*$ is ``closer'' to $v$ than $u$, in the sense that $u^*$ and $v$ have a common prefix which is strictly longer than the common prefix of $u$ and $v$. More formally, by ``deducing'' $u^* \approx v$ from $u \approx v$, using identities of $\mathcal{B}_{\hypo}$, we mean showing that $u \approx u^*$ is a consequence of $\mathcal{B}_{\hypo}$.
		
		The technical part of the proof allows us to show that there is always a way to shuffle some variables of $u$ in such a way that we obtain $u^*$. We show that these variables must occur several times in $u$, thus allowing us to apply the identities of $\mathcal{B}_{\hypo}$ to shuffle $u$ and obtain $u^*$.
		
		By the induction hypothesis, we know that $u^* \approx v$ is a consequence of $\mathcal{B}_{\hypo}$. Thus, we conclude that $u \approx v$ is also a consequence of $\mathcal{B}_{\hypo}$. 
		
		Let $u \approx v$ be a non-trivial identity satisfied by $\hypo$. Let $\mathcal{X}$ be the set of variables which occur in $u$ and $v$. Since $u \approx v$ is a non-trivial identity, we must have $u = w x u'$ and $v = w y v'$, for some words $w,u',v' \in \mathcal{X}^*$ and variables $x,y$ such that $x \neq y$. Notice that $u'$ and $v'$ cannot be the empty word, otherwise, we would have $u= wx$ and $v = wy$, which contradicts the fact that $c(u) = c(v)$.
		
		On the other hand, since $c(xu') = c(yv')$, we have that $y$ occurs in $u'$. Thus, to distinguish the leftmost $y$ in $u'$, we have that
		\[
		x u' = u_1 a y u_2,
		\]
		for some variable $a$ and words $u_1$ and $u_2$, such that $y$ does not occur in $u_1$ and $a \neq y$. Once again, since $c(xu') = c(yv')$, we have that $a$ occurs in $v'$. Thus, to distinguish the leftmost $a$ in $v'$, we have that 
		\[
		v' =v_1 a v_2,
		\]
		for some words $v_1$ and $v_2$, such that $a$ does not occur in $v_1$. To sum up, we have that 
		\[
		u = w u_1 a y u_2 \quad \text{and} \quad v = w y v_1 a v_2,
		\]
		where $y$ does not occur in $u_1$ and $a$ does not occur in $v_1$.
		
		Notice that $u$ admits $ay$ as a subsequence, hence, $v$ must also do so. Thus, either $a$ occurs in $w$, or $y$ occurs in $v_2$, since $a$ does not occur in $v_1$. But if $y$ occurs in $v_2$, then it must also occur in $u_2$, since $c(xu') = c(yv')$ and $y$ does not occur in $u_1$.
		
		On the other hand, notice that $v$ admits $ya$ as a subsequence, hence, $u$ must also do so. Thus, either $y$ occurs in $w$, or $a$ occurs in $u_2$, since $y$ does not occur in $u_1$.
		
		As such, we have four possible cases to look at: 
		
		\textbf{Case 1.} Both variables $y$ and $a$ occur in $w$. Then, we can deduce the word $w u_1 y a u_2$ from $u$, by applying the identity \eqref{idR}, renaming $x$ to $a$.
		
		\textbf{Case 2.} Both variables $y$ and $a$ occur in $u_2$. Then, we can deduce the word $w u_1 y a u_2$ from $u$, by applying the identity \eqref{idL}, renaming $x$ to $a$.
		
		\textbf{Case 3.} Variable $y$ occurs in both $w$ and $u_2$. Then, we can deduce the word $w u_1 y a u_2$ from $u$, by applying the identity \eqref{idM}, renaming $x$ to $y$ and $y$ to $a$.
		
		\textbf{Case 4.} Variable $a$ occurs in both $w$ and $u_2$. Then, we can deduce the word $w u_1 y a u_2$ from $u$, by applying the identity \eqref{idM}, renaming $x$ to $a$.
		
		Observe that we can repeatedly apply this reasoning until we obtain a word of the form
		\[
		u^* = w y u'',
		\]
		for some word $u''$, since the only restriction imposed on the variable $a$ was that $a \neq y$. Thus, we have proven that, for any non-trivial identity $u \approx v$ satisfied by $\hypo$, we can obtain a new word $u^*$ from $u$ such that the common prefix of $u^*$ and $v$ is strictly longer than the common prefix of $u$ and $v$, by applying identities of $\mathcal{B}_{\hypo}$ finitely many times. 
		
		By the induction method, we conclude that $u \approx v$ is a consequence of $\mathcal{B}_{\hypo}$, thus proving that $\mathcal{B}_{\hypo}$ is a basis for $\mathbf{V}_{\hypo}$.
	\end{proof}
	
	An immediate consequence of having a finite basis is the following:
	
	\begin{corollary}
		$\mathbf{V}_{\hypo}$ has finite axiomatic rank.
	\end{corollary}
	
	Now, in order to determine the axiomatic rank of $\mathbf{V}_{\hypo}$, we first check if all identities in $\mathcal{B}_{\hypo}$ are necessary in order to obtain a basis for $\mathbf{V}_{\hypo}$. It is easy to see that right-zero semigroups satisfy identities \eqref{idL} and \eqref{idM}, but not the identity \eqref{idR}, and left-zero semigroups satisfy \eqref{idR} and \eqref{idM}, but not \eqref{idL}. Thus, \eqref{idL} and \eqref{idR} are not consequences of the other identities in $\mathcal{B}_{\hypo}$. 
	
	On the other hand, let $S = \{1,a,b,c,0\}$ be a semigroup with multiplication table:
    \begin{center}
		\begin{tabular}{c|ccccc|}
              & 1 & a & b & c & 0 \\
            \hline
            1 & 1 & a & b & c & 0 \\
            a & a & 1 & b & c & 0 \\
            b & b & c & 0 & 0 & 0 \\
            c & c & b & 0 & 0 & 0 \\
            0 & 0 & 0 & 0 & 0 & 0 \\
            \hline
		\end{tabular}
	\end{center}
		
	This example was produced using Mace4~\cite{prover9-mace4}. It consists of a null semigroup $K = \{b,c,0\}$ with a cyclic group $C_2 = \{1,a\}$ acting on it identically from the left and by permuting $b$ and $c$ on the right. Notice that $a^2 = 1$, $ba = c = ac$ and $ca = b = ab$. 
	
	$S$ satisfies \eqref{idL} and \eqref{idR}: If we substitute all variables with elements of $C_2$, we get equality by commutativity. If, after substituting, we have more than one element of $K$ on each side, both sides equal $0$. So we are left with the case of substituting a variable that appears once (either $z$ or $t$) with an element of $K$ and everything else with elements of $C_2$. But since $z$ and $t$ are in the same position on both sides of the identities, and whichever elements of $C_2$ are substituted for $x$ and $y$ commute, the action of the elements of $C_2$ on the element of $K$ gives the same result. However, $S$ does not satisfy \eqref{idM}: Taking the evaluation $\phi$ such that $\phi(y)=c$ and $\phi(x) = \phi(z) = \phi(t) = a$, we get
	\[
        \phi(xzxytx) = aaacaa = c \neq b = aacaaa = \phi(xzyxtx).
	\]	
	
	Since no identity of $\mathcal{B}_{\hypo}$ is a consequence of the other identities also in $\mathcal{B}_{\hypo}$, we conclude that $\mathcal{B}_{\hypo}$ is minimal, in the sense that it does not contain any proper subset which is also a basis for $\mathbf{V}_{\hypo}$.
	
	We now show that the identities \eqref{idL} and \eqref{idR} are required to be in any basis for $\mathbf{V}_{\hypo}$ which contains only identities over an alphabet with four variables:
		
	\begin{proposition}
	\label{proposition:idL_not_consequence}
        The identity \eqref{idL} is not a consequence of the set of non-trivial identities, satisfied by $\hypo$, over an alphabet with four variables, excluding \eqref{idL} itself and equivalent identities.
	\end{proposition}
	
	\begin{proof}	
        Let $X := \{ x,y,z,t \}$ and let $\mathcal{S}$ be the set of all non-trivial identities, satisfied by $\hypo$, over an alphabet with four variables, excluding \eqref{idL} and equivalent identities. Suppose, in order to obtain a contradiction, that \eqref{idL} is a consequence of $\mathcal{S}$. As such, there must exist a non-trivial identity $u \approx v$ in $\mathcal{S}$, and a substitution $\sigma$, such that
        \[
            xyzxty = w_1 \sigma(u) w_2,
		\]
		where $w_1, w_2$ are words over $X$, and $\sigma(u) \neq \sigma(v)$. Notice that $u \approx v$ must be balanced, and that there must be at least two variables occurring in $u$ and $v$, otherwise, $u \approx v$ would be a trivial identity. 
		
		Observe that if the substitution $\sigma$ maps some variables occurring in $u$ and $v$ to the empty word, then $\sigma(u) = \sigma(u')$ and $\sigma(v) = \sigma(v')$, where $u'$ and $v'$ are words obtained by eliminating every occurrence of such variables in $u$ and $v$, respectively. Hence, we have that
        \[
            xyzxty = w_1 \sigma(u') w_2.
		\]
		Notice that $u' \approx v'$ is an identity satisfied by $\hypo$, which cannot be trivial, otherwise we would have $\sigma(u) = \sigma(u') = \sigma(v') = \sigma(v)$. Thus, $u' \approx v'$ is also in $\mathcal{S}$. On the other hand, notice that $\sigma$ does not map any variable occurring in $u'$ and $v'$ to the empty word. As such, any case where $\sigma$ maps any variable to the empty word guarantees the existence of another case where it does not map any variable to the empty word.
		
        Therefore, we can assume, without loss of generality, that $\sigma$ does not map any variable to the empty word. Due to this, and since \eqref{idL} is an identity where $x$ and $y$ occur two times, and $t$ and $z$ each occur one time, we have that each variable occurring in $u \approx v$ can occur at most two times, and only two variables can occur more than one time. Furthermore, by Corollary \ref{corollary:hypo_shortest_identity_n_variables}, which gives us a lower bound for the length of the identities, we have that $u \approx v$ is of length at least $4$. Thus, up to renaming of variables, $x$ and $y$ occur exactly twice in $u \approx v$, and $t$ and $z$ can occur at most one time.
                
        Suppose now, in order to obtain a contradiction, that $w_1 \neq \varepsilon$. Then, since $u \approx v$ is of length at least $4$, we must have $w_1$ of length at most $2$, that is, $w_1$ is either $x$ or $xy$. Therefore, $x$ can occur only once in $\sigma(u)$. But $x$ and $y$ occur twice in $u$, and $\sigma$ does not map any variable to the empty word, hence, there must be at least two variables which occur twice in $\sigma(u)$. However, only $x$ and $y$ occur twice in $xyzxty$. We have reached a contradiction, hence, $w_1 = \varepsilon$. Using a similar argument, we can also conclude that $w_2 = \varepsilon$. Therefore, we have that
        \[
            xyzxty = \sigma(u).
		\]

        As such, we can immediately conclude that only up to three variables occur in $u \approx v$: If $u \approx v$ were to be a four-variable identity, then it would be of length $6$, and $\sigma$ would be simply renaming the variables, thus implying that $u \approx v$ was equivalent to \eqref{idL}, which contradicts our hypothesis.    
                
        Suppose that $u \approx v$ is a two-variable identity. Hence, it is of length $4$ and $x$ and $y$ both occur twice in it. Since $xyzxty = \sigma(u)$, then either $\sigma(x)$ or $\sigma(y)$ must be a single variable, and the other must be of length $2$. Since no variable occurs more than twice in $\sigma(u)$, this implies that three variables occur twice in $\sigma(u)$. But $x$ and $y$ are the only variables which occur twice in $xyzxty$. We have reached a contradiction, hence, $u \approx v$ is not a two-variable identity.
        
        Then, $u \approx v$ must be a three-variable identity. Hence, it is of length $5$, with $x$ and $y$ occurring twice and $z$ occurring once in it. Notice that $\sigma(x)$ and $\sigma(y)$ must be single variables, otherwise, the length of $\sigma(u)$ would be greater than $6$. These variables cannot be $z$ or $t$, since they occur only once in $xyzxty$. Therefore, $\sigma(z)$ must be a factor of $xyzxty$ of length $2$. But neither $x$ nor $y$ can occur in $\sigma(z)$, hence, this factor cannot exist, and subsequently, $\sigma$ cannot exist. 
                
        As such, we can conclude that \eqref{idL} is not a consequence of the set of non-trivial identities, satisfied by $\hypo$, over an alphabet with four variables, excluding \eqref{idL} itself and equivalent identities.
	\end{proof}
	
	Parallel reasoning shows the analogous result for \eqref{idR}:
	
	\begin{proposition}
	\label{proposition:idR_not_consequence}
        The identity \eqref{idR} is not a consequence of the set of non-trivial identities, satisfied by $\hypo$, over an alphabet with four variables, excluding \eqref{idR} itself and equivalent identities.
	\end{proposition}
    
	Therefore, we can conclude that $\mathbf{V}_{\hypo}$ does not admit any basis with only identities over an alphabet with two or three variables. In other words, we have that:
	
	\begin{corollary}
        The axiomatic rank of $\mathbf{V}_{\hypo}$ is $4$.
	\end{corollary}
	
	Another consequence of the previous proposition is the following:
	
	\begin{corollary}
		Any basis for $\mathbf{V}_{\hypo}$ with only identities over an alphabet with four variables must contain the identities \eqref{idL} and \eqref{idR}, or equivalent identities.
	\end{corollary}
	
	Furthermore, since \eqref{idM} is not a consequence of \eqref{idL} and \eqref{idR}, any basis for $\mathbf{V}_{\hypo}$ with only identities over an alphabet with four variables must contain at least three identities, one of which must be either \eqref{idM}, an equivalent identity, or an identity of which \eqref{idM} is a consequence of.	

	
	\section{Acknowledgment}
	\label{section:acknowledgment}
	

    The authors thank the anonymous referee for their careful reading of the paper and many helpful comments, in particular, for the suggestion of Corollary~\ref{corollary:varietal_join}.

	

\providecommand{\bysame}{\leavevmode\hbox to3em{\hrulefill}\thinspace}
\providecommand{\MR}{\relax\ifhmode\unskip\space\fi MR }
\providecommand{\MRhref}[2]{%
	\href{http://www.ams.org/mathscinet-getitem?mr=#1}{#2}
}
\providecommand{\href}[2]{#2}

\end{document}